\documentclass[10pt]{article}
\usepackage[left=1in,top=1in,right=1in]{geometry}
\usepackage[latin1]{inputenc}
\usepackage{latexsym}
\usepackage{graphics}
\usepackage{graphicx}
\usepackage{hyperref}
\usepackage{times}
\usepackage{multimedia}
\usepackage{amsmath}
\usepackage{amssymb}
\usepackage{amsfonts}
\usepackage{mathrsfs}
\usepackage{colortbl}
\usepackage{multicol}
\usepackage{pst-all}
\usepackage{pstricks}
\usepackage{pstcol}

\newpsobject{grilla}{psgrid}{subgriddiv=1,griddots=10,gridlabels=6pt}
\definecolor{gris1}{cmyk}{0,0,0,0.1}
\definecolor{gris2}{cmyk}{0,0,0,0.3}
\definecolor{gris3}{cmyk}{0,0,0,0.4}

\newtheorem{theorem}{Theorem}

\newtheorem{lemma}[theorem]{Lemma}

\newenvironment{proof}[1][Proof]{\noindent\textbf{#1.} }{\ \rule{0.5em}{0.5em}}
\input{tcilatex}

\begin{document}

\title{Sharp Condition for blow-up and global existence in a two species
chemotactic Keller-Segel system in $\mathbb{R}^{2}$}
\author{Elio E. Espejo\thanks{%
Technion - Israel Institute of Technology, 32000 Haifa, Israel; Email:
eespejo@techunix.technion.ac.il}, Karina Vilches\thanks{%
Departamento de Ingenierìa Matemàtica (DIM) and Centro de Modelamiento Matemà%
tico (CMM), Universidad de Chile (UMI CNRS 2807), Casilla 170-3, Correo 3,
Santiago, Chile. Email: kvilches@dim.uchile.cl}, Carlos Conca\thanks{%
Departamento de Ingenierìa Matemàtica (DIM) and Centro de Modelamiento Matemà%
tico (CMM), Universidad de Chile (UMI CNRS 2807), Casilla 170-3, Correo 3,
Santiago, Chile and Institute for Cell Dynamics and Biotechnology: a Centre
for Systems Biology, University of Chile, Santiago, Chile. Email:
cconca@dim.uchile.cl},}
\maketitle

\begin{abstract}
For the parabolic-elliptic Keller-Segel system in $\mathbb{R}^{2}$ it has
been proved that if the initial mass is less than $8\pi/\chi$ global
solution exist and in the case that the initial mass is larger than $%
8\pi/\chi$ blow-up happens. The case of several chemotactic species
introduces an additional question: What is the analog for the critical mass
obtained for the single species system? We find a threshold curve in the
case of two especies case that allows us to determine if the system has
blow-up or has a global in time solution.
\end{abstract}

\section{Introduction}

The Keller-Segel model describes the aggregation of living organisms like
cells, bacteria or amoebae. This is the simplest mechanism of aggregation.
The most famous example in the nature for this type of cells motion is the
Dictyostelium discoideum or Slime mould, this amoebae was discovered by K.
B. Raper in 1935. The slime mould is a unicellular organism that detect a
extracellular signal and transforms it into an intracellular signal. These
signal activates oriented cell movement toward a signal, this is the
aggregation process. The signal is a chemical secreted by themselves, the
chemical is called cyclic Adenosine Monophosphate (cAMP).\bigskip

A classical mathematical model in chemotaxis was introduced by E.F.~Keller
and L.A.~Segel in \cite{KS}. The Keller-Segel model is: 
\begin{equation}  \label{int}
\left. 
\begin{array}{c}
u_{t}=\nabla \cdot (\mu \nabla u-\chi u\nabla v)\text{ \ \ \ \ \ }x\in
\Omega ,\text{ \ \ \ }t>0 \\ 
v_{t}=\gamma \Delta v-\beta v+\alpha u\text{ \ \ \ \ \ \ \ \ \ }x\in \Omega .%
\text{ \ \ \ }t>0,%
\end{array}%
\right.
\end{equation}%
where $u(x,t)$ is the cell density and $v(x,t)$ is the concentration of the
chemical at point $x$ and time $t$ subject to homogeneous Neumann boundary
conditions and positive initial data $u(x,0)=u_{0}$ and $v(x,0)=v_{0}$. In
this model, $\chi $ is the chemotactic sensitivity, $\gamma $ is the
diffusion coefficient of the chemo-attractant and $\mu $ the diffusion
coefficient of the cell density, $\beta $ is the rate of consumption and $%
\alpha $ is the rate of production, all are positive parameters, and $\Omega 
$ $\subset $ $\mathbb{R}^{N}$ has smooth boundary $\partial \Omega .$ It was
conjectured by S.~Childress \& J.K.~Percus~\cite{CH} that in a
two-dimensional domain there exists a critical number $C$ such that if $\int
u_{0}(x)dx<C$ then the solution exists globally in time, and if $\int
u_{0}(x)dx>C$ blow-up happens. For different versions of the Keller-Segel
model the conjecture has been essentially proved, finding the critical value 
$C=8\pi /\chi $; for a complete review of this topic we refer the reader to
the papers \cite{H}, \cite{H2} and the references therein, particularly, 
\cite{BDP}, \cite{Biler}, \cite{JL}, \cite{N} and \cite{VE}.\bigskip

In the case of several chemotactic species a new question arises, namely, 
\textit{Is there a critical curve in the plane of initial masses }$\theta
_{1}\theta _{2}$\textit{\ delimiting on one side global existence and
blow-up on the other side?}. This question was already formulated by G.
Wolansky in \cite{W} and from Theorem 5 of this last paper we readily deduce
the following result\bigskip

\begin{theorem}
Consider the system 
\begin{eqnarray*}
\partial _{t}u_{1} &=&\Delta u_{1}-\chi _{1}\nabla \cdot (u_{1}\nabla v) \\
\partial _{t}u_{2} &=&\mu _{2}\Delta u_{2}-\chi _{2}\nabla \cdot
(u_{2}\nabla v) \\
0 &=&\Delta v+u_{1}+u_{2}-v,
\end{eqnarray*}%
along with Dirichtlet boundary conditions for $v$ and initial radial data: $%
u_{1}(0,\cdot )=\varphi ,$ $u_{2}(0,\cdot )=\psi ,$ $v(0,\cdot )=\phi ,$
with $\varphi ,\psi ,\phi \geq 0$ on the two-dimensional disc of radius $1$.
Further let $\theta _{1},$ $\theta _{2}$ be the total preserved masses of
the chemotactic species. Assume further that%
\begin{equation}
\left. \frac{4\pi \mu \theta _{1}}{\chi _{1}}+\frac{4\pi \theta _{2}}{\chi
_{2}}-\frac{1}{2}(\theta _{1}+\theta _{2})^{2}>0,\text{ \ }\theta _{1}<8\pi
/\chi _{1},\text{ \ \ \ }\theta _{2}<8\pi /\chi _{2}\right.  \label{in}
\end{equation}%
Then for $(u_{1}(0,$\textperiodcentered $),u_{2}(0,$\textperiodcentered $%
))\in Y_{N}$ with 
\begin{equation*}
Y_{N}=\left\{ u_{1},u_{2}:B(0)\rightarrow\mathbb{R}^{+}:\int u_{i}=\theta
_{i},\text{ \ \ }\int_{B_{1}(0)}u_{i}\log u_{i}<\infty \right\}
\end{equation*}%
there exist a global in time classical solution.\bigskip
\end{theorem}

A natural question arise from this last result, What happens in case
inequalities \ref{in} does not hold? Is it still possible to have global
solutions? With regard to this question it is worth to recall here a result
from [C. Conca, E. Espejo, K. Vilches, \cite{CEV}] who considered the
following system in the whole space in two dimensions,: 
\begin{equation}
\left. 
\begin{array}{l}
\partial _{t}u_{1}=\mu \Delta u_{1}-\chi _{1}\nabla \cdot (u_{1}\nabla v) \\ 
\partial _{t}u_{2}=\Delta u_{2}-\chi _{2}\nabla \cdot (u_{2}\nabla v) \\ 
v(x,t)=-\frac{1}{2\pi }\int_{\mathbb{R}^{2}}\log \left\vert x-y\right\vert
\left( u_{1}(y,t)+u_{2}(y,t)\right) dy \\ 
u_{1}(x,0)=u_{10}\geq 0,\text{ }u_{2}(x,0)=u_{20}\geq 0,%
\end{array}%
\right\}  \label{s}
\end{equation}%
where $t\geq 0$, $u_{1}$ and $u_{2}$ are the density variables for the two
different chemotaxis species and $v$ is the chemoattractant, $\chi _{1},\chi
_{2}$, $\mu $ are positive constants and positive initial conditions $%
u_{10},u_{20}$ are given. In this last paper it was proved that if $\theta
_{1},$ $\theta _{2}$ satisfies \textit{any} of the inequalities, 
\begin{equation*}
\frac{4\pi \mu \theta _{1}}{\chi _{1}}+\frac{4\pi \theta _{2}}{\chi _{2}}-%
\frac{1}{2}(\theta _{1}+\theta _{2})^{2}<0,\text{ \ \ }\theta _{1}>\mu \frac{%
8\pi }{\chi _{1}},\text{ \ \ \ }\theta _{2}>\frac{8\pi }{\chi _{2}},
\end{equation*}%
then system \ref{s} can blow-up. For the global existence was proved also in 
\cite{CEV} that the inequalities 
\begin{eqnarray*}
\theta _{1}+\theta _{2} &<&\frac{8\pi }{\chi _{2}},\text{ \ \ \ }\mu <1 \\
\theta _{1}+\theta _{2} &<&\frac{8\pi }{\chi _{2}}\mu ,\text{ \ }\mu >1
\end{eqnarray*}%
guarantees global existence.\bigskip

In the present paper we aim to give a step further improving the results of
global existence from \cite{CEV} and to prove that even in the \textit{%
non-radial case} inequalities \eqref{in} also guarantees global existence
for system \eqref{s}. In consequence we give a generalization of the
threshold number $8\pi /\chi $ for the classical parabolic-elliptic
Keller-segel system in $%
\mathbb{R}
^{2}$ to a curve for the two species system. The global existence in time
results of the present paper along with the blow-up results from \cite{CEV}
are summaries in Figure \ref{Graf30}.

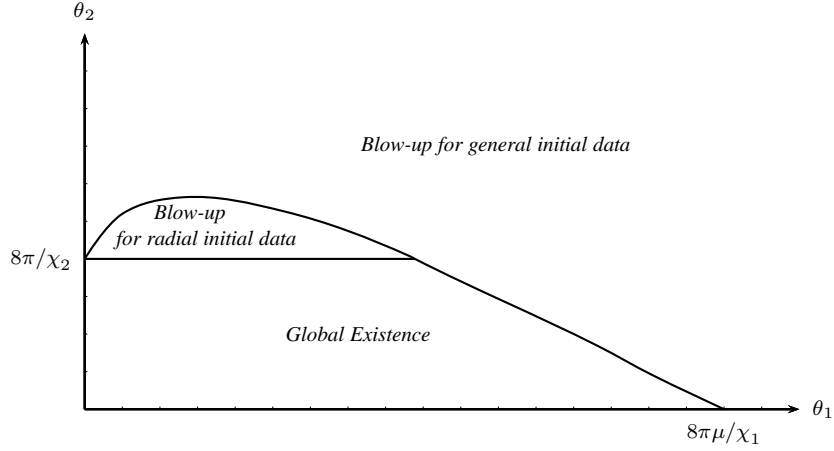
\begin{figure}[htbp]
\label{Graf30}
\par
\begin{center}
\psset{unit=0.5cm} 
\begin{pspicture}(0,-0.5)(19,11)
\pscustom[linestyle=none]{
\pscurve(8.78,4)(13.65,1.68)(15.05,0.93)(17,0)
\psline(17,0)(0,0)
\psline(0,0)(0,4)
\psline(0,4)(8.78,4)}
\pscustom[linestyle=none]{
\pscurve[linestyle=dashed](0,4)(1,5.2)(2.62,5.64)(5.28,5.27)(8.67,4.06)(8.78,4)
\psline(8.78,4)(0,4)}
\pscustom[linestyle=none]{
\pscurve(0,4)(1,5.2)(2.62,5.64)(5.28,5.27)(8.67,4.06)(8.78,4)(13.65,1.68)(15.05,0.93)(17,0)
\psline(17,0)(19,0)
\psline(19,0)(19,10)
\psline(19,10)(0,10)
\psline(0,10)(0,4)}
\psaxes[labels=none,ticksize=1pt]{->}(0,0)(19,10)
\pscurve(0,4)(1,5.2)(2.62,5.64)(5.28,5.27)(8.67,4.06)(8.78,4)(13.65,1.68)(15.05,0.93)(17,0)
\psline(0,4)(8.78,4)
\uput[u](0,10){\footnotesize{$\theta_2$}}
\uput[r](19,0){\footnotesize{$\theta_1$}}
\uput[r](1.5,5.2){\textit{\footnotesize Blow-up}}
\uput[r](0.5,4.5){\textit{\footnotesize for radial initial data}}
\uput[r](5,2){\textit{\footnotesize Global Existence}}
\uput[r](7,7){\textit{\footnotesize Blow-up for general initial data}}
\uput[d](17,0){\footnotesize {$8\pi\mu/\chi_1$}}
\uput[l](0,4){\footnotesize {$8\pi/\chi_2$}}
\end{pspicture}
\end{center}
\caption{Regions of global existence in time and blow-up}
\end{figure}

\section{Preliminaries}

Let us proceed formally to find a free energy functional to our system.
First we write the equation for $u_{1}$ in (\ref{s}) in the form,%
\begin{equation}
\partial _{t}u_{1}=\nabla \cdot u_{1}\nabla \left( \mu \log u_{1}-\chi
_{1}v\right) .  \label{a1}
\end{equation}%
Next, we multiply both sides of (\ref{a1}) by $\mu \log u_{1}-\chi _{1}v$
and integrate to obtain,%
\begin{equation}
\int_{\mathbb{R}^{2}}u_{1t}\left( \mu \log u_{1}-\chi _{1}v\right) dx=\int_{%
\mathbb{R}^{2}}\left( \mu \log u_{1}-\chi _{1}v\right) \nabla \cdot
u_{1}\nabla \left( \mu \log u_{1}-\chi _{1}v\right) dx,  \label{a2}
\end{equation}%
Then using mass conservation and integrating by parts we see that (\ref{a2})
is equivalent to,%
\begin{equation}
\frac{d}{dt}\int_{\mathbb{R}^{2}}\mu u_{1}\log u_{1}dx-\chi _{1}\int_{%
\mathbb{R}^{2}}u_{1t}vdx=-\int_{\mathbb{R}^{2}}u_{1}\left\vert \nabla \left(
\mu \log u_{1}-\chi _{1}v\right) \right\vert ^{2}dx.  \label{a3}
\end{equation}%
Similarly it holds that,%
\begin{equation}
\frac{d}{dt}\int_{\mathbb{R}^{2}}u_{2}\log u_{2}dx-\chi _{2}\int_{\mathbb{R}%
^{2}}u_{2t}vdx=-\int_{\mathbb{R}^{2}}u_{2}\left\vert \nabla \left( \log
u_{2}-\chi _{2}v\right) \right\vert ^{2}dx.  \label{a4}
\end{equation}%
Now we add $\frac{1}{\chi _{1}}$(\ref{a3}) and $\frac{1}{\chi _{2}}$(\ref{a4}%
) to obtain,%
\begin{eqnarray}
&&\frac{d}{dt}\left\{ \int_{\mathbb{R}^{2}}\frac{\mu }{\chi _{1}}u_{1}\log
u_{1}dx+\frac{1}{\chi _{2}}\int_{\mathbb{R}^{2}}u_{2}\log u_{2}dx\right\}
-\int_{\mathbb{R}^{2}}\left( u_{1t}+u_{2t}\right) vdx  \notag \\
&=&-\int_{\mathbb{R}^{2}}u_{1}\left\vert \nabla \left( \mu \log u_{1}-\chi
_{1}v\right) \right\vert ^{2}dx-\int_{\mathbb{R}^{2}}u_{2}\left\vert \nabla
\left( \log u_{2}-\chi _{2}v\right) \right\vert ^{2}dx.  \label{a5}
\end{eqnarray}%
We observe at this point that,%
\begin{eqnarray}
\int_{\mathbb{R}^{2}}\left( u_{1t}+u_{2t}\right) vdx &=&-\frac{1}{2\pi }%
\int_{\mathbb{R}^{2}}\left( u_{1}(x,t)+u_{2}(x,t)\right) _{t}\int_{\mathbb{R}%
^{2}}\log \left\vert x-y\right\vert \left( u_{1}(y,t)+u_{2}(y,t)\right) dydx
\notag \\
&=&-\frac{1}{4\pi }\frac{d}{dt}\int_{\mathbb{R}^{2}\times \mathbb{R}%
^{2}}\left( u_{1}(x,t)+u_{2}(x,t)\right) \left( u_{1}(y,t)+u_{2}(y,t)\right)
\log \left\vert x-y\right\vert dydx  \notag \\
&=&\frac{1}{2}\frac{d}{dt}\int_{\mathbb{R}^{2}}(u_{1}+u_{2})vdx.  \label{a6}
\end{eqnarray}%
In conclusion we deduce from (\ref{a5}) and (\ref{a6}) that,%
\begin{equation}
\frac{d}{dt}\left\{ \int_{\mathbb{R}^{2}}\frac{\mu }{\chi _{1}}u_{1}\log
u_{1}dx+\frac{1}{\chi _{2}}\int_{\mathbb{R}^{2}}u_{2}\log u_{2}dx-\frac{1}{2}%
\int_{\mathbb{R}^{2}}(u_{1}+u_{2})vdx\right\} \leq 0  \label{energy}
\end{equation}%
Result (\ref{energy}) motivate us to define the free energy functional for
system \ref{s} as, 
\begin{equation}
E(t):=\frac{\mu }{\chi _{1}}\int_{\mathbb{R}^{2}}u_{1}\log u_{1}dx+\frac{1}{%
\chi _{2}}\int_{\mathbb{R}^{2}}u_{2}\log u_{2}dx-\frac{1}{2}\int_{\mathbb{R}%
^{2}}u_{1}vdx-\frac{1}{2}\int_{\mathbb{R}^{2}}u_{2}vdx  \label{e}
\end{equation}%
In order to give validity to our calculations we suppose not only that $%
u_{1},u_{2}\in C^{0}(\mathbb{R}^{+},L^{1}(\mathbb{R}^{2}))\cap
L^{2}((0,T);H^{1}(\mathbb{R}^{2}))$ but also that $u_{1}(1+\left\vert
x\right\vert ^{2})$, $u_{2}(1+\left\vert x\right\vert ^{2})$, $u_{1}\log
u_{1}$ and $u_{2}\log u_{2}$ are bounded in $L_{loc}^{\infty }(\mathbb{R}%
^{+},L^{1}(\mathbb{R}^{2}))$. Additionally $\nabla \sqrt{u_{1}}$,$\nabla 
\sqrt{u_{2}}$$\in L_{loc}^{1}(\mathbb{R}^{+},L^{1}(\mathbb{R}^{2}))$ and $%
\nabla v\in L_{loc}^{\infty }(\mathbb{R}^{+}\times \mathbb{R}^{2})$.\newline
Then we have that, 
\begin{equation}
\frac{d}{dt}E(t)=-\frac{1}{\chi _{1}}\int_{\mathbb{R}^{2}}u_{1}\left\vert
\mu \nabla \log u_{1}-\nabla \chi _{1}v\right\vert ^{2}dx-\frac{1}{\chi _{2}}%
\int_{\mathbb{R}^{2}}u_{2}\left\vert \nabla \log u_{2}-\nabla \chi
_{2}v\right\vert ^{2}dx\leq 0.  \label{2}
\end{equation}%
As a consequence of (\ref{2}) and the Hardy-Littlewood-Sobolev inequality (%
\cite{BDP}, \cite{CEV}) was obtained in \cite{CEV} the following entropy
bound,

\begin{theorem}
\label{1} If $u_{1}$ and $u_{2}$ are positive solutions of \eqref{s} on the
interval $[0,T)$ and $\chi _{1}\leq \chi _{2}$ then we have the following
entropy estimates,

\begin{itemize}
\item if $\mu >1$ then

\begin{equation*}
\left( 1-\frac{M\chi _{2}}{8\pi }\right)\int_{0}^{T} \int_{\mathbb{R}%
^{2}}\left( \frac{1}{\chi _{1}}u_{1}(x,t)+\frac{1}{\chi _{2}}%
u_{2}(x,t)\right) \log \left(\frac{1}{\chi _{1}}u_{1}(x,t)+\frac{1}{\chi _{2}%
}u_{2}(x,t)\right) dxdt\leq C_{T};
\end{equation*}

where $C_{T}$ is a constant depending on $T$ and $M=\theta _{1}+\theta _{2}.$

\item If $\mu \leq 1$ then 
\begin{equation*}
\left( 1-\frac{M\chi _{2}}{8\pi \mu }\right) \int_{0}^{T}\int_{\mathbb{R}%
^{2}}\left( \frac{1}{\chi _{1}}u_{1}(x,t)+\frac{1}{\chi _{2}}%
u_{2}(x,t)\right) \log \left( \frac{1}{\chi _{1}}u_{1}(x,t)+\frac{1}{\chi
_{2}}u_{2}(x,t)\right) dxdt\leq\overline{C}_{T}.
\end{equation*}
where $\overline{C}_{T}$ is a constant depending on $T$ and $M=\theta
_{1}+\theta _{2}.$
\end{itemize}
\end{theorem}

Theorem \ref{1} gives bounds for the entropy which is the key tool for the
proof of global existence for system (\ref{s}). In order to improve this
last result it would be desirable to use the HLS inequality for systems
developed by I. Shafrir and G. Wolansky in \cite{SW}. However, as we will
show in section 2, a direct application of this tool to our system do not
give the optimal result that we are looking for. We will show how an
adequate introduction of some auxiliary parameters in (\ref{2}) allows us to
improve the result of global existence obtained in \cite{CEV}, mainly, we
will show that if $\theta _{1},\theta _{2}$ satisfy%
\begin{equation*}
\frac{4\pi \mu \theta _{1}}{\chi _{1}}+\frac{4\pi \theta _{2}}{\chi _{2}}-%
\frac{1}{2}(\theta _{1}+\theta _{2})^{2}\geq 0,\text{ \ }\theta _{1}<\mu 
\frac{8\pi }{\chi _{1}}\text{ },\text{ }\theta _{2}<\frac{8\pi }{\chi _{2}}
\end{equation*}%
then global solutions in time exist.\textbf{\ }No kind of radial symmetry is
assumed.\textbf{\bigskip }

The most fundamental tool used through this paper is the logarithmic
Hardy-Littlewood-Sobolev's inequality for systems, which we proceed to
recall now. Following the notation from \cite{SW} we define the space 
\begin{equation*}
\Gamma _{M}\left( 
\mathbb{R}
^{2}\right) =\left\{ \widetilde{\rho }=\left( \widetilde{\rho _{i}}\right)
_{i\in I}:\widetilde{\rho _{i}}\geq 0,\int\limits_{%
\mathbb{R}
^{2}}\widetilde{\rho _{i}}\left\vert \log \widetilde{\rho _{i}}\right\vert
dx<\infty ,\int\limits_{%
\mathbb{R}
^{2}}\widetilde{\rho _{i}}=M_{i},\int\limits_{%
\mathbb{R}
^{2}}\widetilde{\rho _{i}}\log \left( 1+\left\vert x\right\vert ^{2}\right)
<\infty ,\forall i\in I\right\}
\end{equation*}%
where $M=\left( M_{i}\right) _{i\in I}$ is given. Next we define the
functional $F:\Gamma _{M}\left( 
\mathbb{R}
^{2}\right) \rightarrow R$ by, 
\begin{equation*}
F\left[ \widetilde{\rho }\right] =\sum\limits_{i\in I}\int\limits_{%
\mathbb{R}
^{2}}\widetilde{\rho _{i}}\log \widetilde{\rho _{i}}dx+\frac{1}{4\pi }%
\sum\limits_{j,i\in I}a_{i,j}\int\limits_{%
\mathbb{R}
^{2}}\int\limits_{%
\mathbb{R}
^{2}}\widetilde{\rho _{i}}\left( x\right) \log \left\vert x-y\right\vert 
\widetilde{\rho }_{j}\left( y\right) dxdy.
\end{equation*}%
and the polynomial,%
\begin{equation*}
\Lambda _{J}\left( M\right) =8\pi \sum_{i\in J}M_{i}-\sum_{i,j\in
J}a_{ij}M_{i}M_{j},\text{ \ \ \ \ \ }\forall \varnothing \not=J\subseteq I
\end{equation*}%
Then we have,

\begin{theorem}
{\textbf{Hardy-Litlewood-Sobolev's inequality for systems}}\newline
\label{HLS} Let $A=\left( a_{ij}\right) $ a symmetric matrix such that $%
a_{ij}\geq 0$ for all $i,j\in I$ and $M\in 
\mathbb{R}
_{+}^{n}$. Then:\newline
$\Lambda _{I}\left( M\right) =0$ and 
\begin{eqnarray*}
\Lambda _{J}\left( M\right) &\geq &0,\text{ \ for all }J\subseteq I \\
if\text{ }\Lambda _{J}\left( M\right) &=&0\text{ for some J, then }%
a_{ii}+\Lambda _{J\backslash \{i\}}\left( M\right) >0,\text{ \ }\forall i\in
J.
\end{eqnarray*}%
are necessary and sufficient conditions for the boundedness from below of $F$
on $\Gamma _{M}\left( 
\mathbb{R}
^{2}\right) $.\newline
There exist a minimizer $\rho $ of F over $\Gamma _{M}\left( 
\mathbb{R}
^{2}\right) $ if and only if%
\begin{equation*}
\Lambda _{I}\left( M\right) =0,\text{ \ \ and \ }\Lambda _{J}\left( M\right)
>0,\text{ \ \ \ \ for all }J\nsubseteq I
\end{equation*}
\end{theorem}

\begin{proof}
See \cite{SW}, Th. 4.
\end{proof}

\section{Global existence}

The first result of this section gives us bounds for the entropy
functionals. We achieve our aim through an appropriate use of the HLS
inequality for systems, Th. \ref{HLS}. The main idea of the proof \ read as
follows: \ Given that a direct application of the HLS inequality would
allows us to get bounds \textit{only on a curve} of the $\theta _{1}\theta
_{2}-$plane for the entropies $\int_{%
\mathbb{R}
^{2}}u_{i}(x,t)\log u_{i}(x,t)dx,$\ $i=1,2$, we introduce some parameters
before applying the HLS inequality. This step will allows us `to move', `to
shrink' and `to dilate' this curve in such a way the the full region (\ref%
{region}) is swept and therefore obtain estimation (\ref{entro}) in this
region.\bigskip

We suppose throught this paper that, 
\begin{equation}
\left. 
\begin{array}{c}
u_{10},u_{20}\in L^{1}(%
\mathbb{R}
^{2},(1+\left\vert x\right\vert ^{2})dx) \\ 
u_{10}\log u_{10}\text{, }u_{20}\log u_{20}\in L^{1}(%
\mathbb{R}
^{2},dx)%
\end{array}%
\right\}  \label{lll}
\end{equation}

\begin{lemma}[Lower bound for the entropy functionals]
\label{LowerBoundEntropy}Consider a non-negative weak solution of (\ref{s}),
such that $u_{i}(1+\left\vert x\right\vert ^{2}),$ $i=1,2$ are bounded in $%
L_{loc}^{\infty }(%
\mathbb{R}
^{+},L^{1}(%
\mathbb{R}
^{2})).$ Then we have,%
\begin{equation*}
\int\limits_{%
\mathbb{R}
^{2}}u_{i}\left( x,t\right) \log u_{i}\left( x,t\right) \geq M\log M-M\log %
\left[ \pi \left( 1+t\right) \right] -C,\text{ \ }i=1,2.
\end{equation*}
\end{lemma}

\begin{proof}
In the following $C$ will denote a generic constant. We have from \cite[%
Theorem 1]{CEV} that, 
\begin{equation}
\frac{d}{dt}\int\limits_{%
\mathbb{R}
^{2}}\left( \frac{\mu }{\chi _{1}}u_{1}(x,t)+\frac{1}{\chi _{2}}%
u_{2}(x,t)\right) \left\vert x\right\vert ^{2}dx=\frac{4\theta _{1}}{\chi
_{1}}\mu +\frac{4\theta _{2}}{\chi _{2}}-\frac{1}{2\pi }\left( \theta
_{1}+\theta _{2}\right) ^{2}.  \label{momento}
\end{equation}%
We define, 
\begin{equation*}
n:=\frac{\mu }{\chi _{1}}u_{1}+\frac{1}{\chi _{2}}u_{2};
\end{equation*}%
and%
\begin{equation*}
K:=\frac{4\theta _{1}}{\chi _{1}}\mu +\frac{4\theta _{2}}{\chi _{2}}-\frac{1%
}{2\pi }\left( \theta _{1}+\theta _{2}\right) ^{2}.
\end{equation*}%
Thus we obtain,%
\begin{equation}
\int\limits_{%
\mathbb{R}
^{2}}n(x,t)\left\vert x\right\vert ^{2}dx=Kt+\int\limits_{%
\mathbb{R}
^{2}}n(x,0)\left\vert x\right\vert ^{2}dx\leq C(1+t),  \label{b2}
\end{equation}%
where $C:=\max \left\{ K,\int\limits_{%
\mathbb{R}
^{2}}n(x,0)\left\vert x\right\vert ^{2}dx\right\} .$ From the inequality $%
u_{i}\leq Cn$, where $i=1,2$ and (\ref{b2}) we deduce that,%
\begin{equation*}
\int\limits_{%
\mathbb{R}
^{2}}u_{i}(x,t)\left\vert x\right\vert ^{2}dx\leq C(1+t),\text{ \ }i=1,2
\end{equation*}%
Using the same idea presented in \cite[Lemma 2.5]{BDP}, we observe that,%
\begin{equation}
\begin{array}{ccc}
\int\limits_{%
\mathbb{R}
^{2}}u_{i}\left( x,t\right) \log u_{i}\left( x,t\right) & \geq & \frac{1}{1+t%
}\int\limits_{%
\mathbb{R}
^{2}}u_{i}\left( x,t\right) \left\vert x\right\vert ^{2}-C+\int\limits_{%
\mathbb{R}
^{2}}u_{i}\left( x,t\right) \log u_{i}\left( x,t\right) \\ 
& = & \int\limits_{%
\mathbb{R}
^{2}}u_{i}\left( x,t\right) \log \left[ \frac{u_{i}\left( x,t\right) }{e^{-%
\frac{\left\vert x\right\vert ^{2}}{1+t}}}\right] -C.%
\end{array}
\label{ineq1}
\end{equation}%
Let us now define the variable $\mu $ as, 
\begin{equation*}
\mu \left( x,t\right) =\frac{1}{\pi \left( 1+t\right) }\exp \left( -\frac{%
\left\vert x\right\vert ^{2}}{1+t}\right) .
\end{equation*}%
We obtain then from (\ref{ineq1}) that, 
\begin{eqnarray}
\int\limits_{%
\mathbb{R}
^{2}}u_{i}\left( x,t\right) \log u_{i}\left( x,t\right) &\geq &\int\limits_{%
\mathbb{R}
^{2}}u_{i}\left( x,t\right) \log \left[ \frac{u_{i}\left( x,t\right) }{\mu
\left( x,t\right) }\right] dx-M\log \left[ \pi \left( 1+t\right) \right] -C 
\notag \\
&=&\int\limits_{%
\mathbb{R}
^{2}}\frac{u_{i}\left( x,t\right) }{\mu \left( x,t\right) }\log \left[ \frac{%
u_{i}\left( x,t\right) }{\mu \left( x,t\right) }\right] \mu \left(
x,t\right) dx-M\log \left[ \pi \left( 1+t\right) \right] -C;  \label{b1}
\end{eqnarray}%
where $M=\frac{\mu }{\chi _{1}}\theta _{1}+\frac{1}{\chi _{2}}\theta _{2}.$
Using Jensen's inequality we get from (\ref{b1}) that 
\begin{equation*}
\int\limits_{%
\mathbb{R}
^{2}}u_{i}\left( x,t\right) \log u_{i}\left( x,t\right) \geq M\log M-M\log %
\left[ \pi \left( 1+t\right) \right] -C.
\end{equation*}
\end{proof}

\begin{theorem}[Upper bound for the entropy functionals]
\label{th1} Consider a non-negative weak solution of (\ref{s}), such that $%
u_{i}\left( 1+\left\vert x\right\vert ^{2}\right) ,$ $u_{i}\log u_{i},$ $%
i=1,2$ are bounded in $L_{loc}^{\infty }(%
\mathbb{R}
^{+},L^{1}(%
\mathbb{R}
^{2}))$. If $(\theta _{1},\theta _{2})$ satisfies 
\begin{equation}
\theta _{1}<\frac{8\pi }{\chi _{1}}\mu ;\text{ \ \ }\theta _{2}<\frac{8\pi }{%
\chi _{2}};\text{ \ \ \ \ }8\pi \left( \frac{\theta _{1}}{\chi _{1}}\mu +%
\frac{\theta _{2}}{\chi _{2}}\right) -\left( \theta _{1}+\theta _{2}\right)
^{2}>0;  \label{region}
\end{equation}%
then we have, 
\begin{equation}
\int_{%
\mathbb{R}
^{2}}u_{i}(x,t)\log u_{i}(x,t)dx\leq C,\text{ \ \ \ \ \ \ }  \label{entro}
\end{equation}%
where $i=1,2$ and $C$ is a constant depending only on the parameters $\theta
_{1},\theta _{2},\mu ,\chi _{1}$, $\chi _{2},$ and $E(0)$
\end{theorem}

\begin{proof}
From (\ref{s}) we have that, 
\begin{equation*}
E\left( t\right) \leq E\left( 0\right) ,\forall t>0;
\end{equation*}%
in consequence we have the following estimate, 
\begin{eqnarray*}
&&\frac{\mu }{\chi _{1}}\int\limits_{%
\mathbb{R}
^{2}}u_{1}\left( x,t\right) \log u_{1}\left( x,t\right) dx+\frac{1}{\chi _{2}%
}\int\limits_{%
\mathbb{R}
^{2}}u_{2}\left( x,t\right) \log u_{2}\left( x,t\right) dx \\
&\leq &E\left( 0\right) -\frac{1}{4\pi }\int\limits_{%
\mathbb{R}
^{2}}\int\limits_{%
\mathbb{R}
^{2}}u_{1}\left( x,t\right) u_{1}\left( y,t\right) \log \left\vert
x-y\right\vert dxdy-\frac{1}{4\pi }\int\limits_{%
\mathbb{R}
^{2}}\int\limits_{%
\mathbb{R}
^{2}}u_{1}\left( x,t\right) u_{2}\left( y,t\right) \log \left\vert
x-y\right\vert dxdy \\
&&-\frac{1}{4\pi }\int\limits_{%
\mathbb{R}
^{2}}\int\limits_{%
\mathbb{R}
^{2}}u_{2}\left( x,t\right) u_{1}\left( y,t\right) \log \left\vert
x-y\right\vert dxdy-\frac{1}{4\pi }\int\limits_{%
\mathbb{R}
^{2}}\int\limits_{%
\mathbb{R}
^{2}}u_{2}\left( x,t\right) u_{2}\left( y,t\right) \log \left\vert
x-y\right\vert dxdy.
\end{eqnarray*}%
We introduce positive parameters $a$ and $b$ in the last identity such that 
\begin{equation}
a>\chi _{1},\text{ \ \ }b>\chi _{2}  \label{hyp}
\end{equation}%
in the following way,%
\begin{eqnarray}
&&\frac{\mu }{\chi _{1}}\int\limits_{%
\mathbb{R}
^{2}}u_{1}\left( x,t\right) \log u_{1}\left( x,t\right) dx+\frac{1}{\chi _{2}%
}\int\limits_{%
\mathbb{R}
^{2}}u_{2}\left( x,t\right) \log u_{2}\left( x,t\right) dx  \notag \\
&\leq &E\left( 0\right) -\frac{a^{2}}{\mu ^{2}4\pi }\int\limits_{%
\mathbb{R}
^{2}}\int\limits_{%
\mathbb{R}
^{2}}\frac{\mu u_{1}\left( x,t\right) }{a}\frac{\mu u_{1}\left( y,t\right) }{%
a}\log \left\vert x-y\right\vert dxdy  \notag \\
&&-\frac{ab}{\mu 4\pi }\int\limits_{%
\mathbb{R}
^{2}}\int\limits_{%
\mathbb{R}
^{2}}\frac{\mu u_{1}\left( x,t\right) }{a}\frac{u_{2}\left( y,t\right) }{b}%
\log \left\vert x-y\right\vert dxdy  \notag \\
&&-\frac{ab}{\mu 4\pi }\int\limits_{%
\mathbb{R}
^{2}}\int\limits_{%
\mathbb{R}
^{2}}\frac{u_{2}\left( x,t\right) }{b}\frac{\mu u_{1}\left( y,t\right) }{a}%
\log \left\vert x-y\right\vert dxdy  \notag \\
&&-\frac{b^{2}}{4\pi }\int\limits_{%
\mathbb{R}
^{2}}\int\limits_{%
\mathbb{R}
^{2}}\frac{u_{2}\left( x,t\right) }{b}\frac{u_{2}\left( y,t\right) }{b}\log
\left\vert x-y\right\vert dxdy;  \label{param}
\end{eqnarray}%
By doing so, we can apply now the HLS inequality for systems (Th.\ref{HLS})
to the functions $\mu u_{1}/a$ and $u_{2}/b$ on identity (\ref{param})
getting that, 
\begin{eqnarray*}
&&\frac{\mu }{\chi _{1}}\int\limits_{%
\mathbb{R}
^{2}}u_{1}\left( x,t\right) \log u_{1}\left( x,t\right) +\frac{1}{\chi _{2}}%
\int\limits_{%
\mathbb{R}
^{2}}u_{2}\left( x,t\right) \log u_{2}\left( x,t\right) \\
&\leq &E\left( 0\right) -C+\int\limits_{%
\mathbb{R}
^{2}}\mu \frac{u_{1}\left( x,t\right) }{a}\log \left( \mu \frac{u_{1}\left(
x,t\right) }{a}\right) dx+\int\limits_{%
\mathbb{R}
^{2}}\frac{u_{2}\left( x,t\right) }{b}\log \left( \frac{u_{2}\left(
x,t\right) }{b}\right) dx
\end{eqnarray*}%
where the conditions for the existence of the constant $C$ given by Th. (\ref%
{HLS}) are,%
\begin{eqnarray*}
\Lambda _{\left\{ 1\right\} }\left( M\right) &=&8\pi \mu \frac{\theta _{1}}{a%
}-a^{2}\left( \frac{\theta _{1}}{a}\right) ^{2}\geq 0; \\
\Lambda _{\left\{ 2\right\} }\left( M\right) &=&8\pi \frac{\theta _{2}}{b}%
-b^{2}\left( \frac{\theta _{2}}{b}\right) ^{2}\geq 0; \\
\Lambda _{\left\{ 1,2\right\} }\left( M\right) &=&8\pi \left( \mu \frac{%
\theta _{1}}{a}+\frac{\theta _{2}}{b}\right) -(a^{2}\frac{\theta _{1}}{a}%
\frac{\theta _{1}}{a}+ab\frac{\theta _{1}}{a}\frac{\theta _{2}}{b}+b^{2}%
\frac{\theta _{2}}{b}\frac{\theta _{2}}{b})=0
\end{eqnarray*}%
equivalently,%
\begin{equation}
\left. 
\begin{tabular}{l}
$\theta _{1}\leq \mu \frac{8\pi }{a},\text{ \ \ \ }\theta _{2}\leq \frac{%
8\pi }{b},$ \\ 
$8\pi \left( \mu \frac{\theta _{1}}{a}+\frac{\theta _{2}}{b}\right) -\left(
\theta _{1}+\theta _{2}\right) ^{2}=0,$%
\end{tabular}%
\right\}  \label{H}
\end{equation}%
In conclusion we have proved that condition (\ref{H}) implies,%
\begin{eqnarray}
&&\mu \left( \frac{1}{\chi _{1}}-\frac{1}{a}\right) \int\limits_{%
\mathbb{R}
^{2}}u_{1}\left( x,t\right) \log u_{1}\left( x,t\right) +\left( \frac{1}{%
\chi _{2}}-\frac{1}{b}\right) \int\limits_{%
\mathbb{R}
^{2}}u_{2}\left( x,t\right) \log u_{2}\left( x,t\right)  \notag \\
&\leq &E\left( 0\right) -C+\frac{\theta _{1}\mu }{a}\log \frac{\mu }{a}+%
\frac{\theta _{2}}{b}\log \frac{1}{b}.  \label{te}
\end{eqnarray}%
We have from Lemma \ref{LowerBoundEntropy} that the functionals $\int
u_{i}\log u_{i}dx$ are lower bounded$,$ for $i=1,2.$ On the other side each
of the coefficients of the entropy functionals in (\ref{te}) are positive as
long as $a>\chi _{1}$\ and $b>\chi _{2}.$\ Then we take parameters $a$\ and $%
b$ on the intervals $(\chi _{1}$,$\infty )$\ and $(\chi _{2},\infty )$\
respectively \ We conclude that estimates (\ref{entro}) on region\textbf{\ }(%
\ref{region}) holds.\bigskip
\end{proof}

Boundedness of the entropies in the last Theorem is the main tool that we
will use\ to obtain the following result of global existence.

\begin{theorem}[Global Existence of Weak Solutions]
\label{global} Under assumption (\ref{lll}) and 
\begin{equation}
8\pi \left( \frac{\theta _{1}}{\chi _{1}}\mu +\frac{\theta _{2}}{\chi _{2}}%
\right) -\left( \theta _{1}+\theta _{2}\right) ^{2}>0;  \label{crutial2}
\end{equation}%
\begin{equation}
\theta _{1}<\frac{8\pi }{\chi _{1}}\mu ;\text{ \ \ \ }\theta _{2}<\frac{8\pi 
}{\chi _{2}};  \label{berra}
\end{equation}%
system (\ref{s}) has a global weak nonnegative solution such that 
\begin{equation*}
(1+\left\vert x\right\vert ^{2}+\left\vert \log u_{i}\right\vert )u_{i}\in
L^{\infty }(0,T;L^{1}(%
\mathbb{R}
^{2}))
\end{equation*}%
and%
\begin{equation*}
-\frac{1}{\chi _{1}}\int \int_{\left[ 0,T\right] \times 
\mathbb{R}
^{2}}u_{1}\left\vert \mu \nabla \log u_{1}-\nabla \chi _{1}v\right\vert
^{2}dx-\frac{1}{\chi _{2}}\int \int_{\left[ 0,T\right] \times 
\mathbb{R}
^{2}}u_{2}\left\vert \nabla \log u_{2}-\nabla \chi _{2}v\right\vert
^{2}dx<\infty ,
\end{equation*}
\end{theorem}

Before giving the proof, let us first give some explanations on this result.
Inequality (\ref{crutial2}) corresponds to the interior of a rotated
parabola in the plane $\theta _{1}\theta _{2}.$ Choosing the parameters $\mu
,$ $\chi _{1}$ and $\chi _{2}$ adequately condition (\ref{berra}) may be
relevant or can be simply ignored. Next figure illustrates the two possible
cases:

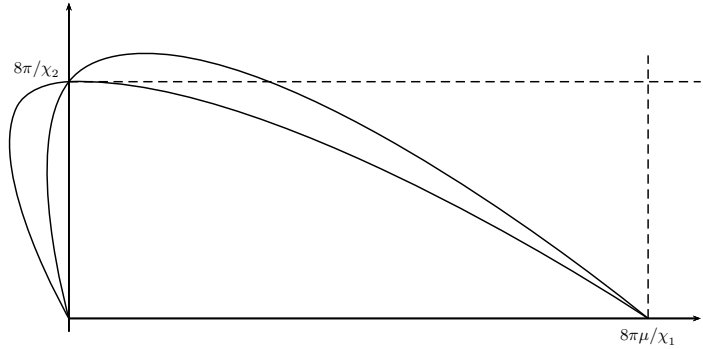
\begin{center}
\scalebox{0.7}{
\psset{yunit=0.5cm}
\begin{pspicture}(-2,-0.5)(12,12)\psaxes[ticks=none,labels=none]{->}(0,0)(0,-0.5)(12,12)
\pscurve(0,0)(-1,8)(0,9)(11,0)
\pscurve(0,0)(0,9)(2,10)(11,0)
\uput[l](0,9.5){\small $8\pi/\chi_2$}
\uput[d](11,0){\small$8\pi\mu/\chi_1$}
\psline[linestyle=dashed](0,9)(12,9)
\psline[linestyle=dashed](11,0)(11,10)
\end{pspicture}}
\end{center}


More precisely we have that,

\begin{itemize}
\item If the parabola,%
\begin{equation}
8\pi \left( \frac{\theta _{1}}{\chi _{1}}\mu +\frac{\theta _{2}}{\chi _{2}}%
\right) -\left( \theta _{1}+\theta _{2}\right) ^{2}=0  \label{t}
\end{equation}
intersects any of the lines $\theta _{1}=8\pi \mu /\chi _{1}$ or $\theta
_{2}=8\pi /\chi _{2}$ in the first quadrant of the $\theta _{1}\theta _{2}$
plane, (which happens exactly when $\chi _{1}<\mu \chi _{2}/2$ or $\chi
_{1}>2\mu \chi _{2}$) and $\theta _{1},$ $\theta _{2}$ satisfies
inequalities (\ref{crutial2}) and (\ref{berra}) then system (\ref{s}) has a
global in time weak solution.

\item However, if the parabola (\ref{t}) do not intersect any of the lines $%
\theta _{1}=8\pi \mu /\chi _{1}$ or $\theta _{2}=8\pi /\chi _{2}$ (when $\mu
\chi _{2}/2\leq \chi _{1}\leq 2\mu \chi _{2})$ in the first quadrant of the $%
\theta _{1}\theta _{2}$ plane, and $\theta _{1},$ $\theta _{2}$ satisfies
inequality (\ref{crutial2}), then system (\ref{s}) has a global in time weak
solution.
\end{itemize}

On the other hand we point out that all of our results are formal so far. In
order to give them rigorousness, we should have a local existence result of
smooth solutions. However we will take another strategy which will allow us
to obtain directly global existence in time of weak solutions with the
corresponding mathematical rigorosity. In order to prove Th.\ref{global} ,
we first modify the convolution kernel $k^{0}(z)=-\frac{1}{2\pi }\log
\left\vert z\right\vert $\ in (\ref{s}), by truncating it around zero. This
last will allows us to get a regularized version of system (\ref{s}) which
is rather easier to work. After proving the existence of global solutions of
this last approximate problem, we look \ for uniform estimates of the
solutions and then a pass to the limit will give us the result of global
existence we are looking for. After getting this result we recover
properties such as mass conservation or the second moment formula by
"testing" properly our weak solution. A similar technique was made in the
one chemotaxis species case (see \cite{BCM}, \cite{BDP}).\bigskip

\begin{proof}[Proof (Sketch)]
For the reader's convenience we divide the proof in four steps giving
special attention where technical difficulties arise in comparison to the
single species case.\newline
Step 1. \textit{Regularization of the system}. We define $K^{\epsilon }$ by $%
K^{\epsilon }\left( z\right) :=K^{1}\left( \frac{z}{\epsilon }\right) ,$
where $K^{1}$ is a radial monotone non-decreasing smooth function
satisfying, 
\begin{equation*}
K^{1}\left( z\right) =\left\{ 
\begin{array}{ccc}
-\frac{1}{2\pi }\log \left\vert z\right\vert & if & \left\vert z\right\vert
\geq 4 \\ 
0 & if & \left\vert z\right\vert \leq 1,%
\end{array}%
\right.
\end{equation*}%
assume also that 
\begin{equation*}
\left\vert \nabla K^{1}\left( z\right) \right\vert \leq \frac{1}{2\pi
\left\vert z\right\vert }
\end{equation*}%
\begin{equation*}
K^{1}\left( z\right) \leq -\frac{1}{2\pi }\log \left\vert z\right\vert ;%
\text{ \ \ }-\Delta K^{1}\left( z\right) \geq 0;\text{, \ \ \ }\forall z\in 
\mathbb{R}
^{2}
\end{equation*}%
for any $z\in 
\mathbb{R}
^{2}.$ .Then we consider the following regularized version of system (\ref{s}%
),%
\begin{equation}
\left\{ 
\begin{array}{l}
\partial _{t}u_{1}^{\epsilon }=\Delta u_{1}^{\epsilon }-\chi _{1}\nabla
\cdot (u_{1}^{\epsilon }\nabla v^{\epsilon }),\text{ \ \ \ }\ t\geq 0,\text{
\ \ \ \ }x\in 
\mathbb{R}
^{2} \\ 
\partial _{t}u_{2}^{\epsilon }=\Delta u_{2}^{\epsilon }-\chi _{2}\nabla
\cdot (u_{2}^{\epsilon }\nabla v^{\epsilon }) \\ 
v^{\epsilon }=K^{\epsilon }\ast \left( u_{1}^{\epsilon }+u_{2}^{\epsilon
}\right) .%
\end{array}%
\right.  \label{el}
\end{equation}%
which we interpret in the distribution sense. Since $K^{\varepsilon }(z)$=$%
K^{1}(\frac{z}{\varepsilon })$ we also have,%
\begin{equation}
\left\vert \nabla K^{\varepsilon }\left( z\right) \right\vert =\frac{1}{%
\varepsilon }\left\vert \nabla K\left( \frac{z}{\varepsilon }\right)
\right\vert \leq \frac{1}{\varepsilon }\frac{1}{2\pi \left\vert
z/\varepsilon \right\vert }=\frac{1}{2\pi \left\vert z\right\vert }.
\label{tt}
\end{equation}%
The proof of global solutions in $L^{2}\left( 0,T;H^{1}(%
\mathbb{R}
^{2}\right) \cap C\left( 0,T;L^{2}(%
\mathbb{R}
^{2}\right) )$ for system (\ref{el}) with initial data in $L^{2}(%
\mathbb{R}
^{2})$ follows essentially the same lines as in \cite[Prop. 2.8]{BDP} and
therefore we omit the proof here\newline
Step 2. \textit{A priori estimates for the approximate solutions }$%
u_{1}^{\varepsilon },$ $u_{2}^{\varepsilon }$ \textit{and} $v^{\varepsilon
}. $ \newline
Consider a solution $\left( u_{1}^{\epsilon },u_{2}^{\epsilon }\right) $ of
the regularized system. If 
\begin{equation*}
\theta _{1}<\frac{8\pi }{\chi _{1}}\mu ;\text{ \ \ }\theta _{2}<\frac{8\pi }{%
\chi _{2}};\text{ \ \ \ \ }8\pi \left( \frac{\theta _{1}}{\chi _{1}}\mu +%
\frac{\theta _{2}}{\chi _{2}}\right) -\left( \theta _{1}+\theta _{2}\right)
^{2}\geq 0,
\end{equation*}%
then, uniformly as $\epsilon \rightarrow 0$, with bounds depending only upon 
$\int\limits_{%
\mathbb{R}
^{2}}\left( 1+\left\vert x\right\vert ^{2}\right) u_{i0}dx$ and $%
\int\limits_{%
\mathbb{R}
^{2}}u_{i0}\log u_{i0}dx$ with $i=1,2$, we have:\newline

\begin{enumerate}
\item[(i)] The function $\left( x,t\right) \rightarrow \left\vert
x\right\vert ^{2}\left( u_{1}^{\epsilon }+u_{2}^{\epsilon }\right) $ is
bounded in $L^{\infty }\left( 
\mathbb{R}
_{loc}^{+};L^{1}\left( 
\mathbb{R}
^{2}\right) \right) $

\item[(ii)] The functions $t\rightarrow \int\limits_{%
\mathbb{R}
^{2}}u_{j}^{\epsilon }\left( x,t\right) \log u_{j}^{\epsilon }\left(
x,t\right) dx$ and $t\rightarrow \int\limits_{%
\mathbb{R}
^{2}}u_{j}^{\epsilon }\left( x,t\right) v^{\epsilon }\left( x,t\right) dx$
are bounded for $j=1,2$

\item[(iii)] The function $\left( x,t\right) \rightarrow u_{j}^{\epsilon
}\left( x,t\right) \log \left( u_{j}^{\epsilon }\left( x,t\right) \right) $
is bounded in $L^{\infty }\left( 
\mathbb{R}
_{loc}^{+};L^{1}\left( 
\mathbb{R}
^{2}\right) \right) $ for $j=1,2$

\item[(iv)] The function $\left( x,t\right) \rightarrow \nabla \sqrt{%
u_{j}^{\epsilon }\left( x,t\right) }$ is bounded in $L^{2}\left( 
\mathbb{R}
_{loc}^{+}\times 
\mathbb{R}
^{2}\right) $ for $j=1,2$

\item[(v)] The function $\left( x,t\right) \rightarrow u_{j}^{\epsilon
}\left( x,t\right) $ is bounded in $L^{2}\left( 
\mathbb{R}
_{loc}^{+}\times 
\mathbb{R}
^{2}\right) $ for $j=1,2$

\item[(vi)] The function $\left( x,t\right) \rightarrow u_{j}^{\epsilon
}\left( x,t\right) \Delta v^{\epsilon }\left( x,t\right) $ is bounded in $%
L^{1}\left( 
\mathbb{R}
_{loc}^{+}\times 
\mathbb{R}
^{2}\right) $ for $j=1,2$

\item[(vii)] The function $\left( x,t\right) \rightarrow \sqrt{%
u_{j}^{\epsilon }\left( x,t\right) }\nabla v^{\epsilon }\left( x,t\right) $
is bounded in $L^{2}\left( 
\mathbb{R}
_{loc}^{+}\times 
\mathbb{R}
^{2}\right) $ for $j=1,2$\newline
The proof of estimates (i)-(vii) follows essentially the same steps as in
the one species case and therefore we remit the reader to \cite[Lema 2.11]%
{BDP}.\newline
In addition we note that from Gagliardo-Nierenberg-Sobolev inequality,%
\begin{equation*}
\left\Vert g\right\Vert _{L^{p}\left( 
\mathbb{R}
^{2}\right) }^{2}\leq C_{GNS}^{\left( p\right) }\left\Vert \nabla
g\right\Vert _{L^{2}\left( 
\mathbb{R}
^{2}\right) }^{2-\frac{4}{p}}\left\Vert g\right\Vert _{L^{2}\left( 
\mathbb{R}
^{2}\right) }^{\frac{4}{p}}\text{, }\forall g\in H^{1}\left( 
\mathbb{R}
^{2}\right) \text{, }\forall p\in \left[ 2,\infty \right)
\end{equation*}%
with $g=\sqrt{u^{\varepsilon }}$ we obtain, 
\begin{equation}
\int\limits_{%
\mathbb{R}
^{2}}\left\vert u_{i}^{\epsilon }\right\vert ^{p/2}dx\leq \left(
C_{GNS}^{\left( p\right) }\right) ^{\frac{p}{2}}\theta _{i}\left\Vert \nabla 
\sqrt{u_{i}^{\varepsilon }}\right\Vert _{L^{2}\left( 
\mathbb{R}
^{2}\right) }^{p-2}  \label{d}
\end{equation}%
for any $p>2$. Estimation (iv) along with (\ref{d}) implies that $%
u_{i}^{\epsilon }$ is uniformly bounded in $\varepsilon $ in $L^{q}\left( 
\mathbb{R}
_{loc}^{+}\times 
\mathbb{R}
^{2}\right) $ for every $q\in \lbrack 1,\infty ).$ \ Therefore we have
proved the following.

\item[(viii)] The function $\left( x,t\right) \rightarrow u_{j}^{\epsilon
}\left( x,t\right) $ is bounded in $L^{p}\left( 
\mathbb{R}
_{loc}^{+}\times 
\mathbb{R}
^{2}\right) $ for $j=1,2,$ $p\geq 1.$
\end{enumerate}

Step 3. \textit{Construction of a strong convergence subsequence in} $L^{p}:$
\ To achieve our aim in this step we will apply the Aubin-Lions compactness
Lemma.\newline
First we get a uniform bound on $\left\Vert \nabla u_{i}^{\epsilon
}\right\Vert _{L_{loc}^{2}((\delta ,T)\times B_{i})}$. We observe that%
\begin{eqnarray}
\frac{d}{dt}\int\limits_{%
\mathbb{R}
^{2}}\left\vert u_{i}^{\epsilon }\right\vert ^{2}dx &=&-2\int\limits_{%
\mathbb{R}
^{2}}\left\vert \nabla u_{i}^{\epsilon }\right\vert ^{2}dx+2\chi
_{1}\int\limits_{%
\mathbb{R}
^{2}}u_{i}^{\epsilon }\nabla u_{i}^{\epsilon }\cdot \nabla v^{\epsilon }dx 
\notag \\
&\leq &-2\int\limits_{%
\mathbb{R}
^{2}}\left\vert \nabla u_{i}^{\epsilon }\right\vert ^{2}dx+2\chi _{1}\left(
\int\limits_{%
\mathbb{R}
^{2}}\left\vert \nabla u_{i}^{\epsilon }\right\vert ^{2}\right) ^{1/2}\left(
\int\limits_{%
\mathbb{R}
^{2}}\left\vert u_{i}^{\epsilon }\right\vert ^{2}\left\vert \nabla
v^{\epsilon }\right\vert ^{2}dx\right) ^{1/2}  \notag \\
&\leq &-2\int\limits_{%
\mathbb{R}
^{2}}\left\vert \nabla u_{i}^{\epsilon }\right\vert ^{2}dx+2\chi _{1}\left(
\int\limits_{%
\mathbb{R}
^{2}}\left\vert \nabla u_{i}^{\epsilon }\right\vert ^{2}\right) ^{1/2}\left(
\int\limits_{%
\mathbb{R}
^{2}}\left\vert u_{i}^{\epsilon }\right\vert ^{3}dx\right) ^{1/3}\left(
\int\limits_{%
\mathbb{R}
^{2}}\left\vert \nabla v^{\epsilon }\right\vert ^{6}dx\right) ^{1/6},
\label{4}
\end{eqnarray}%
where we have used Hölder inequality in the last line. The classical
Gagliardo-Nirenberg-Sobolev inequality along with the Calderon-Zigmund
inequality allow us to conclude that 
\begin{equation}
\left( \int\limits_{%
\mathbb{R}
^{2}}\left\vert \nabla v^{\epsilon }\right\vert ^{6}dx\right) ^{1/6}\leq
C\left( \int\limits_{%
\mathbb{R}
^{2}}\left\vert \Delta v^{\epsilon }\right\vert ^{3/2}dx\right) ^{2/3}.
\label{5}
\end{equation}%
From ineq. (\ref{4}) and (\ref{5}) we conclude that,%
\begin{eqnarray*}
&&\frac{d}{dt}\int\limits_{%
\mathbb{R}
^{2}}\left\vert u_{i}^{\epsilon }\right\vert ^{2}dx \\
&\leq &-2\int\limits_{%
\mathbb{R}
^{2}}\left\vert \nabla u_{i}^{\epsilon }\right\vert ^{2}dx+2C\chi _{1}\left(
\int\limits_{%
\mathbb{R}
^{2}}\left\vert \nabla u_{i}^{\epsilon }\right\vert ^{2}\right) ^{1/2}\left(
\int\limits_{%
\mathbb{R}
^{2}}\left\vert u_{i}^{\epsilon }\right\vert ^{3}dx\right) ^{1/3}\left(
\int\limits_{%
\mathbb{R}
^{2}}\left\vert \Delta v^{\epsilon }\right\vert ^{3/2}dx\right) ^{2/3} \\
&\leq &-2\int\limits_{%
\mathbb{R}
^{2}}\left\vert \nabla u_{i}^{\epsilon }\right\vert ^{2}dx \\
&&+2C\chi _{1}\left( \int\limits_{%
\mathbb{R}
^{2}}\left\vert \nabla u_{i}^{\epsilon }\right\vert ^{2}\right) ^{1/2}\left(
\int\limits_{%
\mathbb{R}
^{2}}\left\vert u_{i}^{\epsilon }\right\vert ^{3}dx\right) ^{1/3}\left(
\left( \int\limits_{%
\mathbb{R}
^{2}}\left\vert u_{1}^{\epsilon }\right\vert ^{3/2}dx\right) ^{2/3}+\left(
\int\limits_{%
\mathbb{R}
^{2}}\left\vert u_{2}^{\epsilon }\right\vert ^{3/2}dx\right) ^{2/3}\right) ,
\end{eqnarray*}%
Integrating respect to $t$ and reordening last inequality we obtain now,%
\begin{multline*}
2\int_{0}^{T}\int\limits_{%
\mathbb{R}
^{2}}\left\vert \nabla u_{i}^{\epsilon }\right\vert ^{2}dxdt \\
-2C\chi _{1}\left\{ \sup_{t\in \left[ 0,T\right] }\left( \int\limits_{%
\mathbb{R}
^{2}}\left\vert u_{i}^{\epsilon }\right\vert ^{3}dx\right) ^{1/3}\left(
\sup_{t\in \left[ 0,T\right] }\left( \int\limits_{%
\mathbb{R}
^{2}}\left\vert u_{1}^{\epsilon }\right\vert ^{3/2}dx\right)
^{2/3}+\sup_{t\in \left[ 0,T\right] }\left( \int\limits_{%
\mathbb{R}
^{2}}\left\vert u_{2}^{\epsilon }\right\vert ^{3/2}dx\right) ^{2/3}\right)
\right\} \int_{0}^{T}\left( \int\limits_{%
\mathbb{R}
^{2}}\left\vert \nabla u_{i}^{\epsilon }\right\vert ^{2}\right) ^{1/2}dt \\
+\int\limits_{%
\mathbb{R}
^{2}}\left\vert u_{i}^{\epsilon }\right\vert ^{2}dx-\int\limits_{%
\mathbb{R}
^{2}}\left\vert u_{i}^{\epsilon }(x,0)\right\vert ^{2}dx\leq 0.
\end{multline*}%
We observe now that,%
\begin{equation*}
\int_{0}^{T}\left( \int\limits_{%
\mathbb{R}
^{2}}\left\vert \nabla u_{i}^{\epsilon }\right\vert ^{2}dx\right)
^{1/2}dt\leq T^{1/2}\left( \int_{0}^{T}\int\limits_{%
\mathbb{R}
^{2}}\left\vert \nabla u_{i}^{\epsilon }\right\vert ^{2}dxdt\right) ^{1/2}
\end{equation*}%
Denoting by $X:=\left\Vert \nabla u_{i}^{\epsilon }\right\Vert
_{L_{loc}^{2}((\delta ,T)\times 
\mathbb{R}
^{2})},$ we conclude from last inequality that for positive constants $a$,$b$
and $c$ we have that,%
\begin{equation*}
aX^{2}-bX+c\leq 0,
\end{equation*}%
in consequence $X:=\left\Vert \nabla u_{i}^{\epsilon }\right\Vert
_{L_{loc}^{2}((\delta ,T)\times 
\mathbb{R}
^{2})}$ is bounded, i.e there exist a constant $C$ such that,%
\begin{equation}
\left\Vert \nabla u_{i}^{\epsilon }\right\Vert _{L_{loc}^{2}((\delta
,T)\times 
\mathbb{R}
^{2})}\leq C.  \label{e1}
\end{equation}%
Now we obtain a bound for $\left\Vert du_{i}^{\varepsilon }/dt\right\Vert
_{L^{2}((0,T);H^{-1}(%
\mathbb{R}
^{2}))}:$\newline
Let $\phi \in H^{1}(%
\mathbb{R}
^{2})$ then we have,%
\begin{eqnarray}
\left\vert \left\langle du_{i}^{\varepsilon }/dt,\phi \right\rangle
\right\vert &=&\left\vert \left\langle \Delta u_{i}-\nabla \cdot \left(
u_{i}\nabla \psi _{i}\right) ,\phi \right\rangle \right\vert \leq \left\vert
\left\langle \nabla u_{i},\nabla \phi \right\rangle \right\vert +\left\vert
\left\langle u_{i}\nabla \psi _{i},\nabla \phi \right\rangle \right\vert 
\notag \\
&\leq &\left\Vert \nabla \phi \right\Vert \left\Vert \nabla u_{i}\right\Vert
+\left\Vert \nabla \phi \right\Vert \left\Vert u_{i}\nabla \psi
_{i}\right\Vert .  \label{8}
\end{eqnarray}%
Thus,%
\begin{equation*}
\left\Vert du_{i}^{\varepsilon }/dt\right\Vert _{H^{-1}(%
\mathbb{R}
^{2})}=\sup_{\left\Vert \phi \right\Vert _{H^{1}(%
\mathbb{R}
^{2})}=1}\left\vert \left\langle du_{i}^{\varepsilon }/dt,\phi \right\rangle
\right\vert \leq \left\Vert \nabla u_{i}^{\varepsilon }\right\Vert _{L^{2}(%
\mathbb{R}
^{2})}+\left\Vert u_{i}^{\varepsilon }\nabla \psi _{i}\right\Vert _{L^{2}(%
\mathbb{R}
^{2})}\leq C.
\end{equation*}%
From the last estimate it follows that,%
\begin{equation}
\left\Vert du_{i}^{\varepsilon }/dt\right\Vert _{L^{2}((0,T);H^{-1}(%
\mathbb{R}
^{2}))}=\left( \int_{0}^{T}\left\Vert du_{i}^{\varepsilon }/dt\right\Vert
_{H^{-1}(%
\mathbb{R}
^{2})}^{2}\right) ^{1/2}\leq C  \label{e2}
\end{equation}%
Compactness: In order to apply the Aubin-Lions Lemma we would like to have
compactness in the containence $L^{2}(%
\mathbb{R}
^{2})\hookrightarrow H^{1}(%
\mathbb{R}
^{2})$, however this is not thrue. However we can take advantage of the fact
that the second moment of each $u_{i}^{\varepsilon },$ $\varepsilon >0,$ is
bounded uniformly in $\varepsilon $. This last fact will allow us to obtain
equicontinuity at infinity for the sequences $\left\{ u_{i}^{\varepsilon
}\right\} _{\varepsilon >0},$ $i=1,2,$ which is the basic ingredient to
translate the compactness result of Rellich-Kondrachov from bounded to
unbounded domains (cf.\cite[Corollary 5.3.1]{Attouch}). We define the spaces 
$B_{0}=H^{1}(%
\mathbb{R}
^{2})\cap \left\{ \left. f\right\vert \text{ \ }\left\vert x\right\vert
^{2}f\in L^{1}(%
\mathbb{R}
^{2})\right\} ,$ $B:=L^{2}(%
\mathbb{R}
^{2})$ and $B_{1}:=B_{0}^{\prime }.$ Let $\left\{ f_{i}\right\} $ and
arbitrary bounded sequence in $B,$ then we have $L^{2}$-equi-integrability
at infinity as the following account shows: 
\begin{eqnarray*}
\int_{\left\{ \left\vert x\right\vert >R\right\} }f_{i}^{2}dx &\leq &\frac{1%
}{R}\int_{\left\{ \left\vert x\right\vert >R\right\} }\left( \left\vert
x\right\vert f_{i}^{1/2}\right) f_{i}^{3/2}dx\leq \frac{1}{R}\left(
\int_{\left\{ \left\vert x\right\vert >R\right\} }\left\vert x\right\vert
^{2}f_{i}dx\right) ^{1/2}\left( \int_{\left\{ \left\vert x\right\vert
>R\right\} }f_{i}^{3}dx\right) ^{1/2} \\
&\leq &\frac{1}{R}\left( \int_{%
\mathbb{R}
^{2}}\left\vert x\right\vert ^{2}f_{i}dx\right) ^{1/2}\left( \int_{%
\mathbb{R}
^{2}}f^{3}dx\right) ^{1/2}
\end{eqnarray*}%
From the Gagliardo-Nirenberg interpolation inequality (cf. \cite[Th. 9.3]%
{Friedman}) with $p=3,r=q=2,j=0,n=2,m=1$ and $a=1/3$, we have that 
\begin{equation}
\left\Vert u\right\Vert _{3}\leq C\left\Vert \nabla u\right\Vert
_{2}^{1/3}\left\Vert u\right\Vert _{2}^{2/3}\text{ \ for all }u\in
C_{0}^{\infty }(%
\mathbb{R}
^{2})  \label{6}
\end{equation}%
Inequality (\ref{6}) holds also in $H^{1}(%
\mathbb{R}
^{2})$, therefore%
\begin{equation*}
\int_{\left\{ \left\vert x\right\vert >R\right\} }f_{i}^{2}dx\leq \frac{C}{R}%
\left( \int_{%
\mathbb{R}
^{2}}\left\vert x\right\vert ^{2}f_{i}dx\right) ^{1/2}\left( \int_{%
\mathbb{R}
^{2}}\left\vert \nabla f_{i}\right\vert ^{2}dx\right) ^{1/2}\left( \int_{%
\mathbb{R}
^{2}}f_{i}^{2}dx\right)
\end{equation*}%
thus,%
\begin{equation}
\lim_{R\rightarrow +\infty }\int_{\left\{ \left\vert x\right\vert >R\right\}
}f_{i}^{2}dx=0\text{ \ uniformly with respect to }f_{i}  \label{e3}
\end{equation}%
From the Rellich-Kondrakov Theorem we obtain the compact inclusion,%
\begin{equation*}
B_{0}\hookrightarrow \hookrightarrow B
\end{equation*}%
Given that $u_{i}^{\varepsilon }$ satisfies (\ref{e1}), (\ref{e2}) and (\ref%
{e3}) we can invoke now the Aubin-Lions-Simon theorem to conclude that $%
u_{i}^{\varepsilon }$ has a subsequence which converge strongly in $%
L^{2}(0,T,B).$ Therefore up to a subsequence we have that,%
\begin{equation}
u_{i}^{\epsilon }\rightarrow u_{i}\text{ a.e. in }%
\mathbb{R}
^{2}\times \lbrack 0,T]  \label{9}
\end{equation}%
We have also proved uniformly boundedness for $\left\Vert u_{i}^{\epsilon
}\right\Vert _{L^{p}(%
\mathbb{R}
^{2})\times \lbrack 0,T]},$ from this, estimation (\ref{9}) and Vitali
theorem we obtain,%
\begin{equation}
u_{i}^{\epsilon }\rightarrow u_{i}\text{ \ strongly in }L^{p}(%
\mathbb{R}
^{2}\times \lbrack 0,T])\text{ for }p\geq 1  \label{11}
\end{equation}%
\newline
Step 4. \textit{Pass to the limit}. We pass now to the limit in the weak
sense to obtain our result of global existence. The most significant
technical difficulty to show that $u_{1},$ $u_{2}$ solved (\ref{s}) arise
with the nonlinear terms. In order to prove that%
\begin{equation}
u_{i}^{\epsilon }\nabla v^{\epsilon }\rightharpoonup u_{i}\nabla v,\text{in }%
D^{\prime }(%
\mathbb{R}
^{+}\times 
\mathbb{R}
^{2}),  \label{big}
\end{equation}%
we notice first that the expression $u_{i}^{\epsilon }\left\vert \nabla
v^{\epsilon }\right\vert $ is integrable as estimate (vii) of part 2 along
with the following estimate shows, 
\begin{eqnarray*}
&&\left( \int_{\left[ 0,T\right] \times 
\mathbb{R}
^{2}}u_{i}^{\epsilon }\left\vert \nabla v^{\epsilon }\right\vert dxdt\right)
^{2}=\left( \int_{\left[ 0,T\right] \times 
\mathbb{R}
^{2}}\sqrt{u_{i}^{\epsilon }}\sqrt{u_{i}^{\epsilon }}\left\vert \nabla
v^{\epsilon }\right\vert dxdt\right) ^{2} \\
&\leq &\int_{\left[ 0,T\right] \times 
\mathbb{R}
^{2}}u_{i}^{\epsilon }dxdt\int_{\left[ 0,T\right] \times 
\mathbb{R}
^{2}}u_{i}^{\epsilon }\left\vert \nabla v^{\epsilon }\right\vert
^{2}dxdt\leq \theta _{i}T\int_{\left[ 0,T\right] \times 
\mathbb{R}
^{2}}u_{i}^{\epsilon }\left\vert \nabla v^{\epsilon }\right\vert ^{2}dxdt,
\end{eqnarray*}%
It follows that we can interpret $u_{i}^{\epsilon }\nabla v^{\epsilon }$ as
an element of $\left( C_{0}^{\infty }\left( 
\mathbb{R}
^{+}\times 
\mathbb{R}
^{2}\right) \right) ^{\prime }$ and therefore it has sense its divergence. 
\newline
In order to prove that $\left\Vert \nabla v^{\varepsilon }\right\Vert
_{L^{r}(%
\mathbb{R}
^{n})}\leq C$ for $r>2$, we recall the Hardy-Littelwood-Sobolev inequality:
For all $f\in L^{p}(%
\mathbb{R}
^{n}),$ $g\in L^{q}(%
\mathbb{R}
^{n})$, $1<p,q<\infty ,$ such that $1/p+1/q+\lambda /n=2$ and $0<\lambda <n$%
, there exist a constant $C=C(p,q,\lambda )>0$ such that%
\begin{equation*}
\left\vert \int_{%
\mathbb{R}
^{n}\times 
\mathbb{R}
^{n}}\frac{1}{\left\vert x-y\right\vert ^{\lambda }}f(x)g(y)dxdy\right\vert
\leq C\left\Vert f\right\Vert _{L^{p}(%
\mathbb{R}
^{n})}\left\Vert g\right\Vert _{L^{q}(%
\mathbb{R}
^{n})}.
\end{equation*}%
Taking the supremum over the ball $\left\Vert g\right\Vert _{L^{q}(%
\mathbb{R}
^{n})}=1$ on both sides of the last inequality we obtain,%
\begin{equation}
\left\Vert \int_{%
\mathbb{R}
^{n}}\frac{1}{\left\vert x-y\right\vert ^{\lambda }}f(x)dx\right\Vert _{L^{%
\frac{q}{q-1}}(%
\mathbb{R}
^{n})}\leq C\left\Vert f\right\Vert _{L^{p}(%
\mathbb{R}
^{n})}  \label{be}
\end{equation}%
In particular%
\begin{equation*}
\left\Vert \int_{%
\mathbb{R}
^{n}}\frac{1}{\left\vert x-y\right\vert }f(x)dx\right\Vert _{L^{\frac{q}{q-1}%
}(%
\mathbb{R}
^{2})}\leq C\left\Vert f\right\Vert _{L^{p}(%
\mathbb{R}
^{2})}\text{ where }1<p,q<\infty ,\text{ and }1/p+1/q+1/2=2.
\end{equation*}%
Thus we have that,%
\begin{eqnarray}
\left\Vert \nabla v^{\varepsilon }\right\Vert _{L^{r}(%
\mathbb{R}
^{n})} &=&\left\Vert \nabla K^{\varepsilon }\ast (u_{1}^{\varepsilon
}+u_{2}^{\varepsilon })\right\Vert _{L^{r}(%
\mathbb{R}
^{n})}  \label{e4} \\
&\leq &\left\Vert \frac{1}{2\pi }\int \frac{1}{\left\vert x-y\right\vert }%
(u_{1}^{\varepsilon }+u_{2}^{\varepsilon })dx\right\Vert _{L^{r}(%
\mathbb{R}
^{n})}\leq C\left( \left\Vert u_{1}^{\varepsilon }\right\Vert _{L^{p}(%
\mathbb{R}
^{2})}+\left\Vert u_{2}^{\varepsilon }\right\Vert _{L^{p}(%
\mathbb{R}
^{2})}\right) \leq C,
\end{eqnarray}%
where we have used step 2 (viii). From $r=\frac{q}{q-1}$ and $1/p+1/q+1/2=2$
we obtain that $\frac{1}{r}=\frac{1}{p}-\frac{1}{2}.$ In \ addition $p\in
(1,2)$ implies that $r\in (2,\infty )$. We conclude that (up to a
subsequence) $\nabla v^{\varepsilon }\rightharpoonup h$, where $h$ is in $%
L^{r}.$ In order to prove that actually $h=\nabla K\ast n$ we have to do
some extra work yet. With this end in mind we propose us now to show that,%
\begin{equation}
\nabla v^{\epsilon }\rightarrow \nabla v\text{ \ \ \ \ }a.e.,  \label{bi}
\end{equation}%
\newline
We have that, 
\begin{eqnarray}
\nabla v^{\epsilon }-\nabla v &=&-\frac{1}{2\pi }\int_{%
\mathbb{R}
^{2}}\frac{x-y}{\left\vert x-y\right\vert ^{2}}\left( \left( u_{1}^{\epsilon
}+u_{2}^{\epsilon }\right) -\left( u_{1}+u_{2}\right) \right) \left(
y,t\right) dy  \notag \\
&&+\int\limits_{\left\vert x-y\right\vert \leq 2\epsilon }\left( \frac{1}{%
\varepsilon }\nabla K^{1}\left( \frac{x-y}{\varepsilon }\right) +\frac{%
\left\vert x-y\right\vert }{2\pi \left\vert x-y\right\vert ^{2}}\right)
\left( u_{1}^{\epsilon }+u_{2}^{\epsilon }\right) (y,t)dy.  \label{ba}
\end{eqnarray}%
From (\ref{11}) and (\ref{be}) we deduce that (up to a subsequence) the
first integral in (\ref{ba}) converges to zero a.e. On the other side,
estimates (\ref{tt}) allows us to conclude that%
\begin{equation*}
\left\vert \int\limits_{\left\vert x-y\right\vert \leq 2\epsilon }\left( 
\frac{1}{\varepsilon }\nabla K^{1}\left( \frac{x-y}{\varepsilon }\right) +%
\frac{\left\vert x-y\right\vert }{2\pi \left\vert x-y\right\vert ^{2}}%
\right) \left( u_{1}^{\epsilon }+u_{2}^{\epsilon }\right) (y,t)dy\right\vert
\leq \int\limits_{\left\vert x-y\right\vert \leq 2\epsilon }\left( \frac{1}{%
\pi \left\vert x-y\right\vert }\right) \left( u_{1}^{\epsilon
}+u_{2}^{\epsilon }\right) (y,t)dy
\end{equation*}%
After taking polar coordinates we observe that last integral converges to $0$
as $\varepsilon \rightarrow 0.$ Therefore we conclude (\ref{bi}). \newline
We obtain therefore from \cite[Prop. 2.46 (i)]{Fonseca} that $\nabla
v_{\varepsilon }\rightharpoonup \nabla K\ast n$ weakly in $L^{r}$ for $r\geq
2.$ Finally we choose conjugate exponents $r=4$ and $p=4/3$ to conclude the
convergence (\ref{11}).
\end{proof}

\section{\textbf{Conclusions and open questions}}

It has been proved in this paper that system (\ref{s}) has a threshold curve
that determines global existence or blow-up. A more difficult task is to
find out if the blow-up has to be simultaneous or not and also to describe
the asymptotics near the blow-up time. A first step in this direction was
given by E. Espejo, A. Stevens, J.J. L. Velazquez in \cite{EEE2}, where it
was shown that the blow-up has to be simultaneous in the radial case. Should
it be the same in the general case? \ Or Should it depend on more specific
information on the initial data? With regard to this point it is worth to
recall that according to \cite{CEV} it is possible to have blow-up even in
the case that the total moment 
\begin{equation}
m(t):=\frac{\pi }{\chi _{1}}\int_{%
\mathbb{R}
^{2}}u_{1}(x,t)\left\vert x\right\vert ^{2}dx+\frac{\pi }{\chi _{2}}\int_{%
\mathbb{R}
^{2}}u_{2}(x,t)\left\vert x\right\vert ^{2}dx  \label{dd}
\end{equation}%
is increasing, that is, when we have%
\begin{equation*}
\frac{4\pi \mu \theta _{1}}{\chi _{1}}+\frac{4\pi \theta _{2}}{\chi _{2}}-%
\frac{1}{2}(\theta _{1}+\theta _{2})^{2}>0
\end{equation*}%
This opens a new possibility: One species could be increasing meanwhile the
other decreases. That is to say the question of a simultaneous blow-up or
not as well as a possible collapse mass separation could eventually not only
depend on the symmetry of the initial data but also on the $L^{1}$ size of
the initial data.\bigskip

On the other side if the parabola, 
\begin{equation}
\frac{4\pi \mu \theta _{1}}{\chi _{1}}+\frac{4\pi \theta _{2}}{\chi _{2}}-%
\frac{1}{2}(\theta _{1}+\theta _{2})^{2}=0,  \label{c}
\end{equation}%
intersects any of the line lines,\ 
\begin{equation}
\text{\ }\theta _{1}=\frac{8\pi }{\chi _{1}}\text{ \ \ or \ }\theta _{2}=%
\frac{8\pi }{\chi _{2}}.  \label{c2}
\end{equation}%
it would be very interesting to study the behavior of system (\ref{s}) on
this lines. Here it is worth to recall that the proof of convergence toward
a delta function at $T=\infty $\ in the one species case, when total mass is
exactly $8\pi /\chi ,$\ uses in a essential way that the second moment is
preserved (see for instance \cite{BCM} ). In contrast for the two species
case, the rotated parabola (\ref{c}) can intersect any of the lines (\ref{c2}%
) and then we obtain threshold lines on which the second moment is not
preserved. A description of the asymptotic behavior in this case seems to
require rather different techniques to those used in the one species case.

\end{document}